\documentclass[12pt]{amsart}
\usepackage{amsthm,amsmath,amsfonts,amssymb,latexsym,amscd,mathrsfs,hyperref,graphicx}
\usepackage{setspace}
\usepackage{cite}
\setdisplayskipstretch{2}

\makeatletter
\renewcommand{\pod}[1]{\allowbreak\mathchoice
  {\if@display \mkern 18mu\else \mkern 8mu\fi (#1)}
  {\if@display \mkern 18mu\else \mkern 8mu\fi (#1)}
  {\mkern4mu(#1)}
  {\mkern4mu(#1)}
}

\theoremstyle{plain}

\newtheorem*{thm*}{Theorem}
\newtheorem{lem}{Lemma}

\makeatletter
\@namedef{subjclassname@2020}{%
  \textup{2020} Mathematics Subject Classification}
\makeatother

\theoremstyle{definition}

\newcommand{\mb}{\mathbb}

\def \a{\alpha} \def \b{\beta}  \def \e{\varepsilon} \def \g{\gamma}    \def \t{\theta} 

\numberwithin{equation}{section}

\renewcommand{\labelenumi}{\setlength{\labelwidth}{\leftmargin}
   \addtolength{\labelwidth}{-\labelsep}
   \hbox to \labelwidth{\theenumi.\hfill}}

\begin{document}
\title{Diophantine approximation with smooth numbers}
\author{Roger Baker\\
I\MakeLowercase{n} m\MakeLowercase{emory of} R\MakeLowercase{ichard} A\MakeLowercase{skey}}
\address{Department of Mathematics\newline
\indent Brigham Young University\newline
\indent Provo, UT 84602, U.S.A}
\email{baker@math.byu.edu}

 \begin{abstract}
Let $\t$ be an irrational number and $\varphi$ a real number. Let $C > 2$ and $\e > 0$. There are infinitely many positive integers $n$ free of prime factors $> (\log n )^C$, such that 
 \[\|\t n + \varphi\| < n^{-\left(\frac 13 - \frac 2{3C}\right) 
 + \e}.\]
 \end{abstract}
\bigskip



\maketitle

\section{Introduction}\label{sec:intro}
Let $\t$ be an irrational number and $\varphi$ a real number. There are many results in the literature concerning small distances $\| \t n + \varphi \|$ of $\t n + \varphi$ from the nearest integer, suggested by the classical inequality. 
 \[\|\t n\| < 1/n\]
for infinitely many $n \in \mb N$. We mention five results here. Harman \cite{har} showed that
 \[\|\t p + \varphi\|  <  p^{-7/22}\]
for infinitely many primes $p$, and Matomaki \cite{mato}  improved the exponent to $-1/3 \, + \e$ in the case $\varphi = 0$. The exponent in \cite{mato} is the limit of current methods; however, Irving \cite{irv} showed that
 \[\|\t n\| < n^{-8/23 + \e}\]
for infinitely many $n$ that are products of two primes. (As usual, $\e$ denotes an arbitrary positive number.)

For square-free integers one can do much better. Heath-Brown \cite{hb} proved that
 \[\|\t n\| < n^{-2/3 + \e}\]
for infinitely many square-free integers $n$.

It is natural to ask what can be done for smooth integers $n$. Yau \cite{yau} has shown (among other results) that if $\t$ is badly approximable, then
 \begin{equation}\label{eq1.1}
\|\t n + \varphi\| < n^{-\frac 14 + \e} 
 \end{equation}
for infinitely many integers $n$ free of prime factors greater than $n^\e$. (Thanks are due to Glyn Harman for pointing out this paper to me.)

In the present paper we consider stronger smoothness conditions on $n$.

 \begin{thm*}
Let $\t$ be an irrational number and $\varphi$ a real number. Let $C > 2$. There are infinitely many $n \in \mb N$ free of prime factors greater than $(\log n)^C$ for which
 \begin{equation}\label{eq1.2}
\|\t n + \varphi\| < n^{-\left(\frac 13 - \frac 2{3C}\right)+\e}.
 \end{equation} 
  \end{thm*}
  
Clearly this implies that in \eqref{eq1.1} we may replace $\frac 14$ by $\frac 13$ (without assuming $\t$ badly approximable.)

Exponential sum bounds for the $y$-smooth numbers up to $x$ given by Fouvry and Tenenbaum \cite{fouten} and Harper \cite{harp} appear to deliver weaker bounds than \eqref{eq1.2}.  Our exponential sum, by contrast, is tailored to this particular application.

We conclude this section with lemmas that will be used in the proof of the theorem. Let $\Psi(x,y)$ denote the number of $n \le x$ that are free of prime factors $> y$. Let $u = \log x/\log y$.

 \begin{lem}\label{lem1}
For $x \ge y \ge 2$, denote by $\a = \a(x,y)$ the unique solution of
 \[\sum_{p \le y}\ \frac{\log p}{p^\a-1} = \log x.\]
For $1 \le c \le y$, we have
 \[\Psi(cx,y) = \Psi(x,y)c^{\a(x,y)}
 \left(1 + O\left(\frac 1u + \frac{\log y}y\right)\right).\]
 \end{lem}
 
 \begin{proof}
This is Corollary 1.7 of \cite{hildten}.
 \end{proof}

 \begin{lem}\label{lem2}
For sufficiently large $x$ and $y=(\log x)^C$, $C > 1$, we have, with $\a = \a(x,y)$,
 \begin{equation}\label{eq1.3}
\left|\a - \left(1 - \frac 1C\right)\right| < \e 
 \end{equation}
and
 \begin{equation}\label{eq1.4}
x^{1 - \frac 1C -\e} \ll \Psi(x,y) \ll x^{1-\frac 1C + \e}. 
 \end{equation}

Implied constants will depend at most on $\e$.
 \end{lem}
 
 \begin{proof}
In \cite{hildten} it is shown that
 \[\a = 1 - \frac{\log(u\log(u+1))}{\log y} +
 O\left(\frac 1{\log y}\right).\]
The inequality \eqref{eq1.3} follows on substituting $y = (\log x)^C$.

The bounds in \eqref{eq1.4} are obtained by combining (1.7) of \cite{hildten} with Theorem 2 of \cite{hildten}.
 \end{proof}
 
 \begin{lem}\label{lem3}
Let $x_1, \ldots, x_N$ be a real sequence with $\|x_n\| \ge M^{-1}$ $(1 \le n \le N)$ where $M$ is a positive integer. Then
 \[\sum_{h=1}^M \Bigg| \sum_{n=1}^N e(hx_n)\Bigg|
 \ge N/6.\]
 \end{lem}
 
 \begin{proof}
See e.g. \cite{bakhar}.
 \end{proof}

 \begin{lem}\label{lem4}
Let $M > 1$, $N > 1$. Let $a_m$ $(m\sim M)$, $b_n$ $(n\sim N)$ be complex numbers with $|a_m| \le 1$, $|b_n| \le 1$. For a real number $\t$ and a fraction $a/q$ with $|\t - a/q| < 1/q^2$, $(a,q) = 1$, we have
 \begin{align*}
&\sum_{m\sim M}\ \sum_{n\sim N} a_mb_ne(mn\t)\\[2mm]
&\quad\ll \Bigg(\sum_{m\sim M} |a_m|^2 \sum_{n\sim N} |b_n|^2\Bigg)^{1/2} \left(\frac{MN}q + M + N + q\right)^{1/2} (\log 2MN)^{1/2}.
 \end{align*}
 \end{lem}

 \begin{proof}
See p. 23 of \cite{har2}
 \end{proof}

We write `$m\sim M$' to indicate $M \le m < 2M$. Let $S_c(x,y)$ denote the set of integers in $[x,cx)$ free of prime factors $> y$.
 \bigskip

 \section{Proof of the Theorem}

Let $a/q$ be a convergent to the continued fraction of $\t$ with $q$ sufficiently large. Thus
 \begin{equation}\label{eq2.1}
\left|\t - \frac aq\right| < \frac 1{q^2} \ , \ (a,q)= 1.
 \end{equation}
We may suppose that $\e$ is sufficiently small, so that
 \[\g : = \frac 13 - \frac 2{3C} - \frac{5\e}6\]
is positive.

Define $x$ by
 \begin{equation}\label{eq2.2}
x^{\frac{1+\g}2} = q. 
 \end{equation}
Suppose that $n \in \mb N$ satisfies
 \begin{equation}\label{eq2.3}
\left\|\frac{an}q + \varphi\right\| \le x^{-\g}, \ n \in S_4(x,y). 
 \end{equation}
Here and below, $y = (\log x)^C$. Then $n$ is free of prime factors $> (\log n)^C$. Moreover,
 \begin{align*}
\|\t n + \varphi\| &\le x^{-\g} + \left|\t - \frac aq\right| \cdot 4x\\[2mm]
&\le x^{-\g} + \frac{4x}{x^{1+\g}} < n^{-\left(\frac 13 - \frac 2{3C}\right)+\e},
 \end{align*}
from \eqref{eq2.1}--\eqref{eq2.3}. Hence it will suffice to show that \eqref{eq2.3} has a solution $n$.

Let $S = S_2(x^{\frac{1+\g}2}, y)$, $J = S_2(x^{\frac{1-\g}2},y)$. By Lemmas \ref{lem1} and \ref{lem2},
 \begin{align}
x^{\frac{1+\g}2\, \left(1-\frac 1C\right) - \frac\e{16}} &\ll |S| \ll x^{\frac{1+\g}2 \left(1 - \frac 1C\right)+\frac \e{16}},\label{eq2.4}\\[2mm]
x^{\frac{1-\g}2\, \left(1-\frac 1C\right) - \frac\e{16}} &\ll |J| \ll x^{\frac{1-\g}2 \left(1 - \frac 1C\right)+\frac \e{16}},\label{eq2.5}
 \end{align}
Here $|\cdots|$ indicates cardinality. We shall show that there are solutions of \eqref{eq2.3} with
 \[n = uv, \ u \in S, \ v \in J.\]

Suppose there are no such solutions of \eqref{eq2.3}. Let
 \[S_h = \sum_{u\in S} \ \sum_{v \in J}
 e\left(\frac{hauv}q\right).\]
Let $H = [x^\g] + 1$. We deduce from Lemma \ref{lem3} and \eqref{eq2.4}, \eqref{eq2.5} that
 \begin{equation}\label{eq2.6}
\sum_{h=1}^H |S_h| \gg x^{1 - \frac 1C - \frac \e 8}.
 \end{equation}

We choose complex numbers $c_h$, $|c_h| \le 1$, such that
 \begin{align}
\sum_{h=1}^H |S_h| &= \sum_{h=1}^H c_h \sum_{u\in S} \ \sum_{v \in J} e\left(\frac{hauv}q\right)\label{eq2.7}\\[2mm]
&= \sum_{u \le 2x^{(1+\g)/2}} \ \sum_{w \le 2x^{(1-\g)/2}H} a_u b_w e\left(\frac{auw}q\right).\notag
 \end{align}
Here
 \begin{align*}
a_u &= \begin{cases}
 1 & \text{if } u \in S\\
 0 & \text{otherwise}\end{cases}\ ,\\[2mm]
 b_w &= \sum_{\substack{hv = w\\
 h \le H,\, v \in J}} c_h. 
 \end{align*}
Thus
 \[\sum_{u\le 2x^{(1+\g)/2}} a_u^2 \ll x^{\frac{1+\g}2 
 \left(1 - \frac 1C\right) +\frac \e{16}},\]
while $b_w \ll x^{\e/8}$ and
 \[\sum_{w\le 2x^{(1-\g)/2}H} b_w^2 \ll x^{\e/8}
 \sum_{w\le 2x^{(1-\g)/2}H} b_w \ll 
 x^{\left(\frac{1-\g}2\right)\left(1 - \frac 1C\right)
 +\g + \frac\e 4}.\]
Here we use \eqref{eq2.4}, \eqref{eq2.5}.

Lemma \ref{lem4} can be applied to the double sum in \eqref{eq2.7} (after a dyadic dissection of the ranges of $u$, $w$) to give the bound
 \begin{align}
\sum_{h=1}^H|S_h| &\ll \left(x^{\frac{1+\g}2\left(1 - \frac 1C\right)+\frac{1-\g}2 \left(1 - \frac 1C\right) + \g + \frac \e 2}\right)^{1/2}\left(\frac{x^{1+\g}}q + x^{\frac{1+\g}2} + x^{\frac{1-\g}2}H+q\right)^{1/2}\label{eq2.8}\\[2mm]
&\ll \left(x^{1 - \frac 1C + \g + \frac \e 2}\right)^{\frac 12}\left(x^{\frac{1+\g}2}\right)^{\frac 12}.\notag
 \end{align}

Combining \eqref{eq2.6}, \eqref{eq2.8}, we have
 \begin{align*}
&x^{1 - \frac 1C - \frac \e 8} \ll x^{\frac 12\left(1 - \frac 1C\right) + \frac{3\g}4 + \frac 14 + \frac \e 4},\\
\frac{3\g}4 &\ge \frac 12\left(1 - \frac 1C\right) - \frac 14 - \frac{2\e}5,\\
\g &\ge \frac 13 - \frac 2{3C} - \frac{4\e}5.
 \end{align*}
This contradicts the definition of $\g$. We conclude that solutions of \eqref{eq2.3} exist. This completes the proof of the theorem.

 \end{document}